\documentclass[11pt]{amsart}
\usepackage{amsmath,amssymb}
\numberwithin{equation}{section}
 \usepackage{xcolor}

\newtheorem{theorem}{Theorem}[section]
\newtheorem{thm}[theorem]{Theorem}

\newtheorem{lem}[theorem]{Lemma}
\newtheorem{prop}[theorem]{Proposition}

\theoremstyle{definition}

\newtheorem{defi}[theorem]{Definition}

\newtheorem{exa}[theorem]{Example}
\newtheorem{quest}[theorem]{Question}

\newtheorem{remark}[theorem]{Remark}
 \newtheorem*{ackn}{Acknowledgements}
 
 \newtheorem*{thmA}{Theorem A} 
 \newtheorem*{thmB}{Theorem B}

 \newcommand{\R}{\mathbb R}
 
 \newcommand{\C}{\mathbb C}
  
 \newcommand{\N}{\mathbb N}

 \newcommand{\e}{\varepsilon}

 \newcommand{\f}{\varphi}
 
 \newcommand{\p}{\psi}

 \newcommand \al {\alpha}
 
 \newcommand \la {\lambda}

 \newcommand \psh {{\rm PSH}}
 \newcommand \PSH {{\rm PSH}}

 \newcommand \MA {{\rm MA}}

 \newcommand{{\CK}}{\hyperref[CK]{Condition (K)}}

 \usepackage{hyperref}
\hypersetup{
    unicode=false,        
    pdftoolbar=true,      
    pdfmenubar=true,       
    pdffitwindow=false,     
    pdfstartview={FitH},    
    pdftitle={YautoKolodziej},    
    pdfauthor={Guedj, Lu},     
    colorlinks=true,       
   linkcolor=blue,          
    citecolor=blue,        
    filecolor=black,      
    urlcolor=blue}



\begin{document}

\title[From Yau to Ko{\l}odziej]{Uniform estimates: from Yau to Ko{\l}odziej}

\author{Vincent Guedj \& Chinh H. Lu}

\address{Institut de Math\'ematiques de Toulouse \& Institut Universitaire de France  \\ Universit\'e de Toulouse \\
118 route de Narbonne \\
31400 Toulouse, France\\}

\email{\href{mailto:vincent.guedj@math.univ-toulouse.fr}{vincent.guedj@math.univ-toulouse.fr}}
\urladdr{\href{https://www.math.univ-toulouse.fr/~guedj}{https://www.math.univ-toulouse.fr/~guedj/}}

\address{LAREMA \& Institut Universitaire de France, Univerist\'e d'Angers, CNRS, F-49000 Angers, France.}

\email{\href{mailto:hoangchinh.lu@univ-angers.fr}{hoangchinh.lu@univ-angers.fr}}
\urladdr{\href{https://math.univ-angers.fr/~lu/}{https://math.univ-angers.fr/~lu/}}
\date{\today}

 \begin{abstract}
  In this note we provide a new and efficient approach to uniform estimates for solutions to complex Monge-Amp\`ere equations,
  as well as for solutions to  geometric PDE’s that satisfy a determinantal majorization.
\end{abstract}

 \maketitle


\section*{Introduction}

 Let $X$ be a compact K\"ahler $n$-dimensional manifold equipped with a K\"ahler form $\omega$.
The complex Monge-Amp\`ere equation is
 \begin{equation} \label{eq:MA}
 \tag{MA(f)}
 \frac{1}{V_{\omega}}(\omega+dd^c \f)^n= f dV_X,
 \end{equation}
 where $V_{\omega}=\int_X \omega^n$ is the volume of $(X,\omega)$,
 $dV_X$ is a smooth probability measure,
  and $f>0$ is a smooth density normalized so that $\int_X f dV_X=1$.
  
  \smallskip
  
  Complex Monge-Amp\`ere equations are ubiquitous in complex geometry, as the Ricci curvature of a K\"ahler form can be expressed by a complex Monge-Amp\`ere operator. 
A crucial step in order to prove the existence of a solution to such equations is to establish a uniform a priori estimate.
A landmark result due to Yau \cite{Yau78} asserts that any solution to the above equation satisfies
$$
{\rm Osc}_X(\f) \leq C (||f||_{p}),
$$
where $C$ depends on an upper-bound on the $L^p$-norm $||f||_p$ of $f$, for some $p>n$
(see Section \ref{sec:hist}). The ${\mathcal C}^0$-estimate is known to be the most difficult,
and often the only one that fails for similar non linear PDEs of (complex) geometric origin.

\smallskip

In connection with the Minimal Model Program, the study of degenerate complex Monge-Amp\`ere equations has become a central
theme in K\"ahler geometry, as advocated by \cite{EGZ09,BEGZ10}.
In this context it becomes essential to establish finer a priori estimates, and understand how these
behave when the density, the cohomology class  $\{\omega\}$ and the complex structure vary.
A key result, which preceded these developments, is due to Ko{\l}odziej \cite{Kol98}, who proved 
$$
{\rm Osc}_X(\f) \leq C (||f||_{p}),
$$
where $p>1$ can be arbitrarily close to $1$. 
Ko{\l}odziej's technique allows one to even replace $||f||_{p}$ by the Luxembourg norm
$||f||_w$ associated to some weaker Orlicz space.
 We refer the reader to Section \ref{sec:kolo} for the precise formulation of the 
(quasi-)optimal {\CK}.

\smallskip

In the last fifteen years several new approaches have been developed to prove uniform a priori estimates 
for a large family of equations, in K\"ahler as well as in hermitian contexts
(see \cite{EGZ08,DP10,TW10,Blo11,DK12, KN15, Szek18,DnGG23,GPT23,GL23,GP24,GL25}).	
Dealing with Orlicz conditions makes the computation 
and geometric dependence of the constants hard to follow.
The first contribution of this note is to provide an efficient 
reduction to Yau's  setting:

\begin{thmA} \label{thmA}
{\it 
Fix $p>n$.
Assume that ${\rm Osc}_X(\f) \leq C (||f||_{p})$  for 
any smooth solution $\f$ of the above complex Monge-Amp\`ere equation \eqref{eq:MA}.
Then 
$$
{\rm Osc}_X(\f) \leq C' (||f||_w),
$$
where $w$ is any weight that satisfies {\CK}, and 
 $||f||_w$ denotes the Luxembourg norm of $f$ with respect to $w$.
}
\end{thmA}

While this result is due to Ko{\l}odziej \cite{Kol98}, our proof is very different
and allows to recover Ko{\l}odziej's optimal condition from Yau's original statement.

\smallskip

Generalizing an idea from \cite{GL25b}, we then show how Theorem A
can be used to obtain uniform a priori estimates for large families of geometric nonlinear PDE's
which have been recently studied by many authors.
 Let $\Gamma \subset \R^n$ denote an open symmetric convex cone with vertex at the origin such that
$$
\R_+^n \subset \Gamma \subset \{ \la \in \R^n, \la_1+\cdots+\la_n >0 \}.
$$
A smooth function $\f: X \rightarrow \R$ belong to ${\Gamma}(X,\omega)$ if 
the eigenvalues $\la_1,\ldots,\la_n$
of $\omega+dd^c \f$ with respect to $\omega$ belong to $\Gamma$ at all points $x \in X$.

\smallskip

We fix $\delta>0$ and a  positive function $g: \Gamma \rightarrow \R_+^*$ such that
\begin{enumerate}
\item $g$ is symmetric (invariance under permutations);
\smallskip
\item $\frac{\partial g}{\partial \la_j}(\la)>0$ for all $\la \in \Gamma$ (ellipticity condition);
\item $g(\la) \geq \delta \left( \pi_{j=1}^n \la_j \right)^{1/n}$ for all $\la \in \R_+^n$ (determinantal majorization).
\end{enumerate}
We consider the nonlinear PDE
\begin{equation}
\tag{NL}
g(\la(\f))=c f^{{1}/{n}},
\end{equation}
where
\begin{itemize}
\item $g$ satisfies the conditions $(1)$-$(2)$-$(3)$ above;
\item $f$ is a fixed smooth positive density;
\item $c>0$ is a normalizing constant determined by the equation;
\item $\f$ is the unknown, a smooth function in ${\Gamma}(X,\omega)$.
\end{itemize}

\begin{thmB}
{\it 
If $\f$ is a smooth solution to \eqref{eq:nl} then 
$$
{\rm Osc}_X(\f) \leq M_0,
$$
where $M_0$ depends on an upper bound on the Luxembourg norm $||f||_{{w}}$,
for any weight $w$ that satisfies {\CK}.
 }
\end{thmB}

The assumptions made on $g$ are slightly less general than the ones considered by 
Szekelyhidi \cite{Szek18} and Guo-Phong-Tong in \cite{GPT23},
but they nevertheless cover a huge variety of complex geometric PDEs as observed by 
Harvey-Lawson \cite{HL23} (see also \cite{ADO22,CX24}).

Our proof applies equally well to the hermitian context and shows that $M_0$ is 
 independent of the complex structure, as we indicate in Remark \ref{rem:herm}.

\begin{ackn} 
 The authors are  supported by the Institut Universitaire de France
and the fondation Charles Defforey. 
This material is based upon work supported by the Clay foundation and the National Science Foundation under Grant No. DMS-1928930, while V.G. was in residence at the 
SLMath in Berkeley, California, 
as part of the Special Geometric Structures and Analysis program.
\end{ackn}

 \section{Basic facts on quasi-plurisubharmonic functions}
 
 In the whole article we fix $X$ a compact complex manifold of  dimension $n \in \N^*$,
 $\omega_X$ a hermitian form, and let $dV_X$ be a smooth probability measure.

\subsection{Skoda-Zeriahi theorem} \label{sec:skoda}

  A function is quasi-plurisub\-harmonic (quasi-psh) if it is locally given as the sum of  a smooth and a psh function.   
  
  \smallskip
  
  Fix $\omega$ another smooth hermitian $(1,1)$-form.

\begin{defi}
 Quasi-psh functions
$\f:X \rightarrow \R \cup \{-\infty\}$ satisfying
$
\omega_{\f}:=\omega+dd^c \f \geq 0
$
in the weak sense of currents are called $\omega$-plurisubharmonic ($\omega$-psh for short).
We let $\PSH(X,\omega)$ denote the set of all $\omega$-plurisubharmonic functions which are not identically $-\infty$.  
\end{defi}

Note that constant functions are $\omega$-psh functions.
A ${\mathcal C}^2$-smooth function $u$ has bounded Hessian, hence $\e u$ is
$\omega$-psh if $0<\e$ is small enough and $\omega$ is  positive.

The set $\PSH(X,\omega)$ is a closed subset of $L^1(X)$, 
for the $L^1$-topology.
Subsets of  $\omega$-psh functions enjoy strong compactness properties,
in particular 
$$
\PSH_0(X,\omega):=\{ u \in\PSH(X,\omega), \, \sup_X u = 0 \}
$$
 is compact in  $L^r$ for all $r \geq 1$.
They also enjoy uniform integrability properties:

  \begin{thm} \label{thm:skoda}
Fix $A>0$ and $\omega$ a K\"ahler form such that
$\int_X \omega \wedge \omega_X^{n-1} \leq A$.
There exists $\alpha=\alpha(n,A)>0$ such that for all 
$\f \in PSH_0(X,\omega)$,
$$
\int_{X} \exp(-\alpha \f) dV_{X} \leq C,
$$
where $C=C(\alpha,n,A)>0$ is independent of $\omega$ and $\f$.
\end{thm}

We refer the reader to \cite[Theorem 8.11]{GZbook} for this result, as well as for
further basic properties of $\omega$-psh functions.

\subsection{The $\omega$-psh envelope}

Upper envelopes of (pluri)subharmonic functions are classical objects in Potential Theory. 
We consider here envelopes of $\omega$-psh functions; they have shown to be extremely
useful in the study of degenerate complex Monge-Amp\`ere equations
(see e.g. \cite{GL25,GL23}).

\begin{defi} \label{def:usual}
Given a Lebesgue measurable function $h:X \rightarrow \R$, we define the $\omega$-psh envelope of $h$ by
$$
P_{\omega}(h) := \left(\sup \{ u \in \psh (X,\omega) ; u \leq h  \, \, \text{in} \, \, X\}\right)^*,
$$
where the star means that we take the upper semi-continuous regularization. 
\end{defi}

  The following  result is  due to \cite{Ber19,Tos18,CZ19}:

 \begin{thm} \label{thm:envelope}
If $h$ is ${\mathcal C}^{\infty}$-smooth then 
\begin{itemize}
\item $P_{\omega}(h)$ is ${\mathcal C}^{1,1}$-smooth;
\item $(\omega+dd^c P_{\omega}(h))^n$ is concentrated on the contact $\{ P_{\omega}(h)=h\}$;
\item one has $\omega+dd^c h \geq 0$ on $\{ P_{\omega}(h)=h\}$ and
$$
(\omega+dd^c P_{\omega}(h))^n=1_{\{ P_{\omega}(h)=h\}} (\omega+dd^c h)^n.
$$
\end{itemize}
\end{thm}

     \subsection{Condition (K)}  \label{CK}
      \label{sec:kolo}
   
 
   \begin{defi}
   A weight $w: \R^+ \rightarrow  \R^+$  satisfies {\CK} if it is
   convex increasing and there exists  an increasing function $h$ such that 
 $$
 w(t) \stackrel{+\infty}{\sim} t (\log t)^n (h \circ \log \circ \log t)^n 
 \; \; \; \text{ and } \; \;
 \int^{+\infty} \frac{dt}{h(t)}<+\infty.
 $$
   \end{defi}
   
 If $w$ satisfies {\CK}, it has been proved by Ko{\l}odziej (see \cite[Theorem 2.5.2]{Kol98}) that 
 any smooth solution $\f$ of
  $$
  \frac{1}{V_{\omega}}(\omega+dd^c \f)^n=fdV_X,
  $$
satisfies a uniform a priori bound on ${\rm Osc}_X(\f)$ that depends on an explicit upper bound on
the Luxembourg norm $||f||_{w}$, where 
$$
||f||_{w}=\inf \left\{ r>0 , \; \int_X w \left( \frac{|f|}{r} \right) dV_X < w(1) \right\}.
$$

\begin{exa}
Among the most classical families of  weights, we mention
\begin{itemize}
\item  $u_p(t)=t^p$ satisfies {\CK} whenever $p>1$;
\item  $v_p(t)=t (\log (t+10))^p$ satisfies {\CK} if and only if $p>n$;
\item $w_p(t)=t (\log (t+10))^n (\log \circ \log (t+10))^p$ satisfies {\CK} if and only if $p>n$.
\end{itemize}
\end{exa}
  
 In  Section \ref{sec:pfthmA}
 we provide a new and simpler proof  of Ko{\l}odziej's result, that goes through a reduction to the case of
 $L^p$ densities, where $p$ is as large as desired (e.g. $p>n$ so as to apply Yau's original method).

    \subsection{Various approaches to uniform estimates} \label{sec:hist}
 
 Yau's proof of his famous $L^{\infty}$-a priori estimate \cite{Yau78} goes through a   Moser iteration process,
 as we recall in Section \ref{sec:yau}.
The method did not apply when the right hand side is too degenerate,
but we show in Section \ref{sec:yau} how to extend it.

\smallskip

An important generalization of Yau's estimate has been provided by Ko{\l}odziej \cite{Kol98} using pluripotential techniques.
These  have been further generalized in \cite{EGZ09,EGZ08,DP10,BEGZ10}
in order to deal with less positive or collapsing families of cohomology classes on K\"ahler manifolds.
This  method  has been partially extended to the hermitian setting in \cite{DK12,KN15,KN19,KN21,GL23}.

\smallskip

Blocki has provided a different approach in \cite{Blo11} based on the Alexandroff-Bakelman-Pucci maximum principle.
This has been pushed further by Szekelehydi in \cite{Szek18}. 
It requires the reference form $\omega$ to be strictly positive.

\smallskip

A PDE proof  has been provided by 
Guo-Phong-Tong \cite{GPT23,GP24}, extending ideas from  \cite{CC21,WWZ21}. 
This technique applies to large families of geometric PDEs  and has led to
striking geometric applications (diameter and non collapsing bounds, uniform controls
on Sobolev constants heat kernels, and positivity of  
K\"ahler-Einstein metrics, see \cite{GPSS24,GPSS23,GT24,Szek24,V24}).

\smallskip

We provide in Section \ref{sec:nl} an alternative proof of the uniform estimate 
for many such equations, using  $\omega$-psh envelopes. 
The latter lie at the heart of yet another approach
to the $L^{\infty}$-a priori estimate, that we have developed in \cite{GL23,GL25}.

\section{From Yau to Ko{\l}odziej} \label{sec:pfthmA}

 \subsection{Proof of Theorem A}
 
 Fix $\omega$ a K\"ahler form, set $V_{\omega}=\int_X \omega^n$, and 
 let $\varphi \in PSH_0(X,\omega)$ be a smooth solution to the complex Monge-Amp\`ere equation
 $$
 \frac{1}{V_{\omega}}(\omega+dd^c \f)^n= f dV_X,
 $$
 where $f>0$ a smooth density such that $\int_X f dV_X=1$.

 \subsubsection*{Step 1} \label{sec:yau}
 
 We first recall Yau's celebrated a priori estimate, as simplified by Kazdan and Aubin-Bourguignon,
 following the presentation of Blocki \cite{Blo12}.
 
 \begin{thm} \label{thm:yau}
Fix $p>n$. There exists $C_1=C_1(n,p,C_S(\omega),||f||_{p})>0$ such that
$
{\rm Osc}_X(\f) \leq C_1 ,
$
where $C_S(\omega)$ denotes   the Sobolev constant of $\omega$.
 \end{thm}

 \begin{proof}
Translating and rescaling we can assume that $V_{\omega}=1$ and $\sup_X \varphi=-1$.
Replacing $f$ by $\tilde{f}$ if necessary, we further assume that $dV_X=\omega^n$.
Observe that $||\varphi||_r \leq||\varphi||_s$ if $r\leq s<\infty$. 
We have
  $$(f-1)\omega^n=(\omega+dd^c\varphi)^n-\omega^n=dd^c\varphi\wedge
    T,$$
where
  $T=\sum_{j=0}^{n-1}(\omega+dd^c\varphi)^j\wedge\omega^{n-1-j}\geq
     \omega^{n-1}.$
Stokes theorem yields, for $r>0$,
{\small
\begin{eqnarray*}
\int_X (-\varphi)^r (f-1)\omega^n
&=& \int_X (-\varphi)^r dd^c\varphi\wedge T=-\int_X d(-\varphi)^r\wedge d^c\varphi \wedge T  \\
&=& r\int_X (-\varphi)^{r-1}d\varphi\wedge d^c\varphi\wedge T
     \geq r \int_X (-\varphi)^{r-1}d\varphi\wedge d^c\varphi\wedge
        \omega^{n-1}\\
     &=& \frac{4r}{(r+1)^2}\int_X d(-\varphi)^{\frac{r+1}{2}}
       \wedge d^c(-\varphi)^{\frac{r+1}{2}}\wedge\omega^{n-1}
 \end{eqnarray*}
 }
 so that
$$
 \int_X (-\varphi)^r (f-1)\omega^n \geq \frac{c_n r}{(r+1)^2}||\nabla (-\varphi)^{\frac{r+1}{2}}||_2^2.
$$
Using H\"older inequality we obtain
$$
||\nabla (-\varphi)^{\frac{r+1}{2}}||_2 \leq c_n' \sqrt{1+||f||_p} \frac{r+1}{\sqrt{r}}
 \left( \int_X (-\f)^{rq} \omega^n \right)^{\frac{1}{2q}},
$$
where $q$ denotes the H\"older conjugate of $p$.
 Using Sobolev inequality 
$$
  ||v||_{\frac{2n}{n-1}}\leq C_S(\omega)\left(||v||_2+||\nabla v||_2\right),
$$
we infer
$$
||\f||_{\frac{n(r+1)}{n-1}} \leq 
C^{\frac{2}{r+1}} (1+||f||_p)^{\frac{2}{r+1}}
\left\{ ||\f||_{r+1}+ \left( \frac{r+1}{\sqrt{r}} \right)^{\frac{2}{r+1}}  
\left( ||\f||_{rq} \right)^{\frac{r}{r+1}} \right\}.
$$

\smallskip

Assume now that $p<+\infty$ hence $q>1$. We are going to use the previous inequality for 
$r \geq \frac{1}{q-1}$, so we can assume that $r+1 \leq rq$. Our normalization ensures that
$||\f||_{rq} \geq 1$, so the previous inequality yields the following
$$
||\f||_{\frac{n(r+1)}{n-1}} \leq  \left(C' [r+1] \right)^{\frac{1}{r+1}}  ||\f||_{rq}.
$$
This inequality is only meaningful when $\frac{n(r+1)}{n-1}>rq$. This requires
  $\frac{n}{n-1}>q$, which is equivalent to $p>n$.

\smallskip

We can now apply Moser's iteration scheme: set
  $$
  r_0:=\frac{1}{q-1}
 \; \;   \text{ and } \; \;
   r_k:=\frac{n}{n-1} \left( \frac{r_{k-1}}{q}+1 \right).
  $$
We obtain
  $$
  ||\varphi||_\infty=\lim_{k\to\infty}||\varphi||_{r_k}\leq
     ||\varphi||_{\frac{1}{q-1}} \prod_{k=0}^\infty \left(C' \left[ \frac{r_{k-1}}{q}+1 \right]\right)^{\frac{1}{\frac{r_{k-1}}{q}+1 }}.
     $$
This infinite product converges as $r_k$ grows geometrically towards $+\infty$.
The conclusion follows as $||\f||_{\frac{1}{q-1}}$ is  bounded by compactness of $PSH_0(X,\omega)$.
 \end{proof}

 \subsubsection*{Step 2}
 
 Fix $n<\tilde{p}<p$. Using H\"older inequality and Theorem \ref{thm:skoda} we can fix
  $0<\gamma=\gamma(n,p,A)=\alpha(n,p,A) \frac{p-\tilde{p}}{p\tilde{p}}$ 
  so that $\tilde{f}=e^{-\gamma \f} f \in L^{\tilde{p}}$ with
 $$
 ||\tilde{f}||_{\tilde{p}} \leq ||f||_p \left(\int_X  e^{-\alpha \f} dV_X \right)^{\frac{p-\tilde{p}}{p\tilde{p}}}
 \leq C_1 ||f||_p.
 $$
 
 We let $\rho$ denote the smooth   solution in $PSH_0(X,\omega)$ of the equation
 $$
 \frac{1}{V_{\omega}}(\omega+dd^c \rho)^n= \frac{\tilde{f} dV_X}{\int_X \tilde{f} dV_X}=\tilde{F} dV_X,
 $$
 where $\int_X \tilde{f} dV_X=\int_X e^{-\gamma \f} f dV_X \geq \int_X  f dV_X=1$ since $\f \leq 0$, hence
 $$
  ||\tilde{F}||_{\tilde{p}} \leq  ||\tilde{f}||_{\tilde{p}} \leq C_1 ||f||_p.
 $$
 It follows from Theorem \ref{thm:yau} that $\rho$ is uniformly bounded.
 We use this function to establish the following useful technical result:
 
 \begin{lem} \label{lem:trick}
 Fix $p>n$, $0<a<1$ and $b>0$. Assume that there is a smooth function  $v \in PSH_0(X,\omega)$ 
 such that 
 $$
  \frac{1}{V_{\omega}} (\omega+dd^c \f)^n \leq a  \frac{1}{V_{\omega}} (\omega+dd^c v)^n+b f dV_X.
 $$
 Then $||\f||_{\infty} \leq C$, where $C=C(n,p,a,b,||f||_p,||v||_{\infty})$.
 \end{lem}
 
 The result was proved in \cite[Theorem 3.3]{DDL21} using generalized Monge-Amp\`ere capacities. We provide below a different and simpler proof. 
 
 \begin{proof}
 Set $MA(\cdot):= \frac{1}{V_{\omega}} (\omega+dd^c \cdot)^n$.
 Fix $0<\delta<1$ such that $\la=a/\delta^n \in (0,1)$ and
 consider 
 $$
  u=\delta v+(1-\delta)\rho - C,
 $$
 where $\rho$ is the function constructed above and the value of the constant $C=C(a,b,\delta,n)$ is specified hereafter.
It suffices to show that $u \leq \f \leq 0$. 

\smallskip
 
 Observe that $u$ is a smooth function in $PSH(X,\omega)$ such that
 $$
 MA(u)\geq
  \delta^n  MA(v) + (1-\delta)^n MA(\rho).
 $$
  Now   $\int_X \tilde{f} dV_X \leq ||\tilde{f}||_{\tilde{p}} \leq C_1 ||f||_p$ 
  and $v,\rho \leq 0$ hence
 $$
 MA(\rho) \geq \frac{1}{C_1 ||f||_p} e^{-\gamma \f} f dV_X \geq e^{\gamma C-\log [C_1 ||f||_p]} e^{-\gamma(\f-u)} f dV_X
 $$
 and we eventually obtain
 \begin{flalign*}
 	\lambda MA(u) & \geq   \lambda \delta ^n  MA (v) +  \lambda (1-\delta)^n e^{\gamma C-\log [C_1 ||f||_p]} e^{-\gamma(\f-u)} f dV_X \\
 	&= a MA(v)+b e^{-\gamma(\f-u)} f dV_X,
 \end{flalign*}
 by choosing 
 $$
 C=\frac{\log \left[ C_1 ||f||_p \, a^{-1}(1-\delta)^{-n} b\delta^{n} \right]}{\gamma}.
 $$
 On the set $(\f<u)$ we thus obtain
 $$
 MA(\f) \leq a MA(v) + b f dV_X \leq \la MA(u).
 $$
 
Since $\la<1$ we infer that the set $(\f<u)$ is empty, i.e. $u \leq \f \leq 0$ is uniformly bounded below.
 Indeed, if this set is not empty then at a point $x_0$ realizing the minimum of $\f-u$ we have $dd^c \f \geq dd^c u$ hence
 $MA(\f)(x_0) \geq MA(u)(x_0)$. Since $u(x_0)<\varphi(x_0)$ and $\MA(\varphi)\leq \lambda \MA(u)$ in $(\varphi<u)$, we get a contradiction. 
  \end{proof}

 \subsubsection*{Step 3} 
 
 We finally use Lemma \ref{lem:trick} to reduce Ko{\l}odziej's criterion to Theorem \ref{thm:yau}.
 Let $w: \R^+ \rightarrow [1,+\infty)$  be a convex increasing weight satisfying {\CK}, i.e.
 $w(t) \sim t (\log t)^n (h \circ \log \circ \log t)^n$ with 
$h$ increasing and $\int^{+\infty} h(t)^{-1}dt<+\infty$,  such that $w\circ f\in L^1(X,dV_X)$.
 It is possible to find $t_0>>1$ such that 
\[
w(t)\geq \frac{2}{3} t (\log t)^n (h\circ \log \circ \log t)^n, \; \forall t\geq t_0. 
\]
Without loss of generality, we can assume that $t_0=10$, and $w$ is smooth.

 Let  $v$ be a smooth function in $PSH_0(X,\omega)$ such that
 $$
 MA(v)=\frac{1}{V_{\omega}} (\omega+dd^c v)^n=c \,  w \circ f  dV_X,
 \; \text{ where } \;
 c^{-1}=\int_X w \circ f dV_X.
 $$
 We let $\chi: \R^- \rightarrow \R^-$ be a smooth convex increasing function such that 
 $0 \leq \chi' \leq 1$. Thus $\chi \circ v$ is a smooth function such that
 $$
 \omega+dd^c \chi \circ v=[1-\chi' \circ v] \, \omega+\chi' \circ v \, [\omega+dd^c v] +\chi'' \circ v \, dv \wedge d^c v > 0,
 $$
 i.e. $\chi \circ v \in PSH(X,\omega)$ and $MA(\chi \circ v) \geq (\chi' \circ v)^n MA(v)$.
 
 \smallskip

 Using Theorem \ref{thm:skoda} we can fix $\al>0$ small enough so that 
 $e^{-\al v} \in L^p$ for $p>n$.
 We also fix $B>\log 10$ (to be specified below).
 In $\{ \log f <-\al v +B \}$ we obtain
 $$
 MA(\f) =f dV_X \leq e^B e^{-\al v} dV_X.
 $$
 
 On the other hand in $\{ \log f  \geq -\al v +B \}$ we obtain
 $$
 MA(\chi \circ v) \geq c (\chi' \circ v)^n w \circ f dV_X 
 \geq c \left[ \chi' \left( \frac{-\log f +B}{\al} \right) \right]^n w \circ f dV_X.
 $$
We choose $\chi$ so that for all  $s > \log 10$,
 $$
 s \, \chi' \left( \frac{-s +B}{\al} \right) h \circ \log s \geq \left(\frac{2}{c}\right)^{1/n}.
 $$
 It follows that 
 $$
 MA(\chi \circ v) \geq \frac{4}{3} f dV_X=\frac{4}{3} MA(\f)
\; \text{  in } \;
\{ \log f  \geq -\al v +B \} \subset   \{\log f >\log 10 \},
$$
 hence
 $$
MA(\f) \leq \frac{3}{4} MA(\chi \circ v)+e^B e^{-v} dV_X
\; \; \text{ on } \; \; X.
 $$
 The conclusion follows from Lemma \ref{lem:trick} if we show that $\chi \circ v$ is uniformly bounded.
 
 \smallskip
 
 Set 
 $$
 b(x)=-\chi\left( \frac{-e^{x}+B}{\al} \right).
 $$ 
 The  condition on $\chi'$ becomes
 $b'(x)h(x) \geq c'=\alpha \left(\frac{2}{c}\right)^{1/n}$ for all  $x > \log \log 10$.
 We thus choose $b'(x)=c' h^{-1}(x)$. The integrability condition on $h$ ensures that
 $b$, hence $\chi$, is uniformly bounded. 
 
  Observe finally that 
   $\chi'(0)=\frac{\al}{Bh(B)}$. Since $h$ is increasing, we can ensure that $0 \leq \chi'(0) \leq 1$
 by choosing $B$ large enough.  $\Box$

 \begin{remark}
 The proof shows that $||\chi||_{\infty} \leq C(1/c)^{1/n} = C \left( \int_X w \circ f dV_X \right)^{1/n}$ 
 for some uniform  $C$, whose dependence on geometric constants is easy to follow.
 \end{remark}

 \subsection{Radial examples}

   We assume here that the measures and quasi-psh functions under consideration are smooth 
   in $X \setminus \{ p \}$. We choose a local chart biholomorphic to the unit ball $B$ of $\C^n$,
   with $p$ corresponding to the origin. We further assume that the data
   are invariant under the group    $U(n,\mathbb C)$ near $p$.
   The  singularity type then only depends on the local behavior near $p$.
   
 We thus consider the local situation of 
  a psh function $v=\chi \circ L$, 
  where $\chi:\R^- \rightarrow\R^-$ is a 
 smooth  strictly convex increasing function with $\chi'(-\infty)=0$ and $L:z \in  B \rightarrow \log |z|^2 \in \R^-$.
  A standard computation shows that
\[  dd^c v=(\chi' \circ L) \, dd^c L +(\chi'' \circ L) \, dL \wedge d^c L\]
  and
$$
  (dd^c v)^n =c_n \frac{(\chi' \circ L)^{n-1}\chi'' \circ L}{|z|^{2n}} dV
=c_n \big[(\chi'^{n-1}\chi'' e^{-n\cdot})\circ L\big] \, dV,
 $$
 where $dV$ denotes here the euclidean volume form.
 These local computations can easily be globalized by using a max construction.
  We let the details to the reader and focus on the asympototic behavior at $0$
  of our local computations.
  
 \smallskip

  We  want to test the optimality of {\CK}. We can 
  assume without loss of generality that $\chi$ does not decrease too fast at infinity,
 imposing that $\chi'(t),\chi''(t) \ge e^{Ct}$ for some $C>0$. 
This guarantees that $\log f \sim -\log |z|$, so that  
the  measure  $(dd^c v)^n=f \, dV$ satisfies {\CK} if and only if
   $$
  \int_{-\infty} \chi''(t) (\chi'(t))^{n-1}  |t|^n (h \circ \log |t|)^{n} dt <+\infty.
  $$
   for some increasing function $h$ such that $\int^{+\infty} h(t)^{-1} dt <+\infty$.

   \begin{exa} \label{exa:optimal}
For $\chi(t)=-\log \circ \log \circ \log (-t+10000)$ we obtain near $-\infty$
$$
\chi'(t) \sim \frac{1}{(-t) \log (-t) \log \circ \log (-t)}
\; \; \text{ and } \; \;
\chi''(t) \sim \frac{1}{(-t)^2 \log (-t) \log \circ \log (-t)}
$$
hence 
  $$
  \int_{-\infty} \chi''(t) (\chi'(t))^{n-1}  |t|^n (h \circ \log |t|)^{n} dt 
  \sim \int^{+\infty} \frac{h(s)^n ds}{s^n (\log s)^n} <+\infty,
  $$
  if we choose $h(s)=s^{1-1/n}$ and $n \geq 2$. This provides an unbounded example s.t.
  $$
  \int_X f (\log [f+10])^n (\log \circ \log [f+100])^{n-1} dV <+\infty.
  $$
   \end{exa}

On the other hand for $h(t)=t^{1+\e/n}$ with $\e>0$, {\CK} becomes
$$
  \int_X f (\log [f+10])^n (\log \circ \log [f+100])^{n+\e} dV <+\infty
  \Longrightarrow 
  v \text{ is bounded.}
  $$
  Thus there is a small gap in the previous example, which motivates the following:

 \begin{quest}
 Is {\CK} optimal ? More precisely can one find $0<\e<1$ and some mildly unbounded psh function $v$
 such that  $MA(v)=f dV$ with 
 $$
 \int_X f (\log [f+10])^n ( \log \circ \log [f+100])^{n-\e} dV<+\infty ?
 $$
 \end{quest}

 We observe that one can not improve Example \ref{exa:optimal} in the radial case:
assume that $p>n-1$,  $v=\chi \circ L$ and $(dd^c v)^n=f dV$ with 
 $$
  \int_{-\infty}^0 \chi''(t) (\chi'(t))^{n-1}  |t|^n (\log |t|)^{p} dt <+\infty.
  $$
  An integration by parts shows that
  $$
  \int_{-\infty}^0   (\chi'(t))^{n} |t|^{n-1} (\log |t|)^{p} dt \lesssim
  \int_{-\infty}^0 \chi''(t) (\chi'(t))^{n-1}  |t|^n (\log |t|)^{p} dt <+\infty,
  $$
while it follows from H\"older inequality that 
{\small
$$
\int_{-\infty}^0 \chi'(t) dt \leq 
\left(  \int_{-\infty}^0   (\chi'(t))^{n} |t|^{n-1} (\log |t|)^{p} dt \right)^{1/n}
\left(  \int_{-\infty}^0 |t| (\log |t|)^{\frac{p}{n-1}} dt \right)^{1-1/n} <+\infty.
$$
}
Thus $\chi(-\infty)>-\infty$, hence $v$ is bounded whenever $p>n-1$.

\section{Other nonlinear equations} \label{sec:nl}

The technique introduced by Guo-Phong-Tong in 
 \cite{GPT23}
applies to a large family of equations. We show here that most of them 
can be treated by a simple reduction to the Monge-Amp\`ere equation, following 
an idea from \cite{GL25b}.

\subsection{Quasi-subharmonic functions}

\subsubsection{$\omega$-subharmonic functions}

 Fix $U\subset X$ an open set.
  A function $h\in {\mathcal C}^2(U,\mathbb R)$ is called harmonic with respect to $\omega$ if 
 \[
 dd^c h \wedge \omega^{n-1} = 0 \; \text{pointwise in }\; U. 
 \]

A function $u : U \rightarrow \mathbb R \cup \{-\infty\}$ is  subharmonic with respect to $\omega$
if it is upper semicontinuous and for every harmonic function $h$ in $D\subset U$, one has
 \[
 u \leq h \; \text{on} \; \partial D \Longrightarrow u \leq h \; \text{in}\; D. 
 \]
 
When $u\in {\mathcal C}^2(U)$, it  is subharmonic with respect to $\omega$ iff $dd^c u \wedge \omega^{n-1}\geq 0$ pointwise.
If $u$ is subharmonic with respect to $\omega$ and in $L^1$ then $dd^c u \wedge \omega^{n-1}\geq 0$ in the sense of distributions. 
Conversely the latter condition implies that $u$ coincides 
  a.e. with a unique function $\hat{u}$ subharmonic with respect to $\omega$.

   \begin{defi}
 	A function $\f:U \rightarrow \R \cup \{-\infty\}$ is quasi-subharmonic with respect to $\omega$ if locally
 	 $\f= u+\rho$ where $u$ is subharmonic with respect to $\omega$ and $\rho$ is smooth. 
 	We say that  $\f$ is $\omega$-subharmonic if it is quasi-subharmonic with respect to $\omega$
 	and  $(\omega+dd^c \f) \wedge \omega^{n-1}\geq 0$ in the sense of distributions on $X$. 
 	
 	We let $SH(X,\omega)$ denote the set of integrable $\omega$-subharmonic functions in $X$.
   \end{defi}
   
We will need the following classical facts:
\begin{itemize}
\item  if $u\in {\mathcal C}^2$ then $u\in SH(X,\omega)$ if and only if $(\omega+dd^c u)\wedge \omega^{n-1} \geq 0$ pointwise;
\item the subset $SH_0(X,\omega)=\{ u \in SH(X,\omega), \; \sup_X u=0 \}$ is compact in $L^1$.
\end{itemize}   
   We refer the reader to \cite{GZbook} for further properties of $\omega$-subharmonic functions:

\subsubsection{The set $\Gamma(X,\omega)$}

Let $\Gamma \subset \R^n$ denote an open symmetric convex cone with vertex at the origin such that
$$
\R_+^n \subset \Gamma \subset \{ \la \in \R^n, \la_1+\cdots+\la_n >0 \}.
$$
In this section we fix $\delta>0$ and a  positive function $g: \Gamma \rightarrow \R_+^*$ such that
\begin{enumerate}
\item $g$ is symmetric (invariance under permutations);
\smallskip
\item $\frac{\partial g}{\partial \la_j}(\la)>0$ for all $\la \in \Gamma$ (ellipticity condition);
\item $g(\la) \geq \delta \left( \pi_{j=1}^n \la_j \right)^{1/n}$ for all $\la \in \R_+^n$ (determinantal majorization).
\end{enumerate}

\begin{defi}
A smooth function $\f: X \rightarrow \R$ belong to ${\Gamma}(X,\omega)$ if 
the eigenvalues $\la_1,\ldots,\la_n$
of $\omega+dd^c \f$ with respect to $\omega$ belong to $\Gamma$ at all points $x \in X$.
\end{defi}

Observe that $PSH(X,\omega) \subset \Gamma(X,\omega) \subset SH(X,\omega)$,
with $\Gamma(X,\omega) = SH(X,\omega)$ when $\Gamma=\{ \la \in \R^n, \la_1+\cdots+\la_n >0 \}$,
while $PSH(X,\omega) \subset \Gamma(X,\omega)$ when $\Gamma=\R_+^n$.

\subsubsection{Determinantal majorization}

We consider the nonlinear PDE
\begin{equation}
\tag{NL}
\label{eq:nl}
g(\la(\f))=c f^{1/n},
\end{equation}
where
\begin{itemize}
\item $g$ satisfies the conditions $(1)$-$(2)$-$(3)$ above;
\item $f$ is a fixed smooth positive density;
\item $c>0$ is a normalizing constant determined by the equation;
\item $\f$ is the unknown, a smooth function in ${\Gamma}(X,\omega)$.
\end{itemize}

The assumptions made on $g$ are slightly less general than the ones considered in \cite{Szek18,GPT23,GP24},
but they nevertheless cover a huge family of complex geometric PDE's as observed by \cite{HL23}
(see also \cite{ADO22,CX24}).

\subsection{Proof of Theorem B}

\begin{thm}
Fix $w$ a weight that satisfies {\CK}.
If $\f$ is a smooth solution to \eqref{eq:nl} then 
$$
{\rm Osc}_X(\f) \leq M_0,
$$
where $M_0$ depends on an upper bound on the Luxembourg norm $||f||_{{w}}$.
\end{thm}

The statement is a generalization of the main result of \cite{GPT23} as far as the condition on $f$ is concerned 
(see also \cite{Q24}). The estimates turn out to be also independent of the complex structure 
(see Remarks \ref{rem:fam} and \ref{rem:herm}).
As with Theorem A, the proof of this result is somehow more interesting than the statement itself.

\begin{proof}
We can normalize $\f$ so that $\sup_X \f=0$, and the problem boils down to show that 
$\f$ is uniformly bounded below. Consider 
$$
\p=P_{\omega}(\f)=\sup \{ u \in PSH(X,\omega), \; u \leq \f \},
$$
the   $\omega$-psh envelope of $\f$. It suffices to show that $\p$ is uniformly bounded below.

\smallskip

Let ${\mathcal C}=\{\p=\f\}$ denote the contact set. 
The function $\p$ is ${\mathcal C}^{1,1}$-smooth and
its Monge-Amp\`ere measure $MA(\p)$ is concentrated on ${\mathcal C}$.
It follows therefore from Theorem \ref{thm:envelope} and
the determinantal majorization that
$$
\delta^n MA(\p) = \delta^n 1_{{\mathcal C}} MA(\f)
\leq  1_{{\mathcal C}} g(\la(\f))^n \omega^n
\leq  c^n f \omega^n,
$$
since $\omega+dd^c \f \geq 0$ on the contact set.

 We can now invoke Theorem A to conclude that
 $
 {\rm Osc}_X (\p) \leq M_0,
 $
 where $M_0$ depends on an upper-bound on the Luxembourg norm 
 $$
 || \frac{c^n}{\delta^n} f ||_w= \frac{c^n}{\delta^n} ||  f ||_w.
 $$
 
 It remains to obtain a uniform control on $\sup_X \p$.
 Let $w^*$ denote the Legendre transform of the convex weight $w$, and let
 $h=(w^*)^{-1}$ denote the inverse  of  the bijective function $w^*$.
It follows from  H\"older-Young inequality that
$$
0 \leq h(-\sup_X \p) \leq \int_X h(-\p) MA(\p)
\leq \frac{c^n}{\delta^n }\int_X h(-\f) f \omega^n 
\leq \frac{c^n}{\delta^n } ||f||_{w} \int_X (-\f) \omega^n.
$$
The conclusion follows from compactness in $L^1$ of functions in $\Gamma(X,\omega)$ that
are normalized by $\sup_X \f=0$.
\end{proof}

\begin{remark} \label{rem:fam}
There is a vast literature on uniform estimates for complex Monge-Amp\`ere equations.
 \cite{DnGG23} provides estimates that 
remain uniform as the cohomology class and the complex structure vary. The proof above thus  
provides uniform estimates for holomorphic families of solutions to \eqref{eq:nl} as well.
 \end{remark}
 
 \begin{remark} \label{rem:herm}
The proof of Theorem B is exactly the same in the hermitian and in the K\"ahler setting. The estimates are therefore
also independent of the complex structure, as follows from \cite{Pan23}, 
and provide a generalization of those from \cite{GP24} for those equations that  satisfy a determinantal majorization.
\end{remark}

\subsection{Alternative proof}

In the previous proof of Theorem B we actually need to apply Theorem A to ${\mathcal C}^{1,1}$-smooth functions, rather than 
${\mathcal C}^{\infty}$-smooth ones.
One can reduce to the smooth case by approximation, but we now propose an alternative proof that   
allows one to work with ${\mathcal C}^{\infty}$-smooth objects.

\smallskip

We again normalize $\f$ so that $\sup_X \f=0$,  
let $w^*$ denote the Legendre transform of the convex weight $w$, and let
 $h=(w^*)^{-1}$ denote the inverse  of  the bijective function $w^*$.
H\"older-Young inequality ensures that
$$
\int_X h(-\f) f \omega^n 
\leq  ||f||_{w} \int_X (-\f) \omega^n \leq B
$$
is uniformly bounded.
We fix positive constants $\lambda,M$ so large that
$$
\frac{2}{1+\lambda}<\frac{\delta^n}{3^n c^n}
\; \; \text{ and } \; \;
\frac{2B}{h(M)} \leq \min \left( 1, \frac{\delta^n}{3^n c^n} \right)
\; \; \text{ with } \; \;
\frac{2^n}{3^n} \leq \frac{V_{\omega} }{1+\lambda e^{-M^2}}.
$$

Let $\psi\in \PSH(X,\omega)$ be the smooth solution to the Monge-Amp\`ere equation 
$$
(\omega+dd^c \psi)^n = b_M \beta_M(\varphi) f \omega^n,
$$ 
where $\sup_X \psi=0$,
$b_M$ is a positive normalization constant
and 
$$
\beta_M(\varphi)= \frac{1}{1+ \lambda e^{M(\varphi+M)}}
$$
is a smooth function taking values in $(0,1]$.  

\smallskip

The constant $b_M$  can be estimated as follows.
Observe first that 
\begin{flalign*}
\int_X \beta_M(\varphi) f \omega^n &= 	\int_{\{\varphi\leq -M\}}   f\omega^n + \int_{\{\varphi>-M\}} \frac{1}{1+\lambda} f\omega^n \\
& \leq \frac{1}{h(M)} \int_X h(-\varphi) f\omega^n + \frac{1}{1+\lambda}  \\
&\leq \frac{B}{h(M)}  + \frac{1}{2}\frac{\delta^n}{3^n c^n} \leq \frac{\delta^n}{3^n c^n},
\end{flalign*}
hence $b_M \geq V_{\omega} 3^n c^n \delta^{-n}$.
On the other hand
\begin{flalign*}
\int_X \beta_M(\varphi) f\omega^n & \geq 	 \int_{\{\varphi>-M\}} \frac{1}{1+\lambda e^{M^2}} f\omega^n \\
& = \frac{1}{1+\lambda e^{M^2}}  \left(1- \int_{\{\varphi\leq -M\}} f\omega^n\right) \\
& \geq  \frac{1}{1+\lambda e^{M^2}}  \left(1- \frac{1}{h(M)}\int_X h(-\varphi) f\omega^n \right)\\
&\geq   \frac{1}{2(1+\lambda e^{M^2})}
\end{flalign*}
yields a uniform upper bound for $b_M$. 
It follows therefore from Theorem A  that $\psi$ is uniformly bounded.  

On the set $(\varphi<\psi-2M)$ we have $\varphi\leq -2M$, hence 
$\beta_M(\f) \geq \frac{1}{1+\lambda e^{-M^2}}$ and
\begin{eqnarray*}
g(\lambda(\p))^n \omega^n \geq  \delta^n (\omega+dd^c \p)^n 
&=& \delta^n b_M \beta_M(\f) f \omega^n \\ 
& \geq & V_{\omega} 3^n c^n \frac{1}{1+\lambda e^{-M^2}} f \omega^n \\
& \geq & 2^n c^n f \omega^n=2^n  g(\lambda(\f))^n \omega^n.
\end{eqnarray*}
Therefore $2 g(\lambda(\f)) \leq g(\lambda(\p))$ on the set $(\varphi<\psi-2M)$
and it follows from the domination principle below that this set is empty, i.e.
 $\varphi \geq \psi-2M$ is uniformly bounded.  $\Box$

\begin{remark}
This  proof can be applied to the Dirichlet problem for such equations
in strongly pseudoconvex domains, in which case the regularity of plurisubharmonic envelopes is
more delicate. We thank A.Zeriahi for pointing this out.
\end{remark}

We have used the following domination principle which is a direct consequence of the classical minimum principle. 

\begin{prop}[Domination principle]
	Assume $\varphi,\psi \in \mathcal{C}^2(X,\mathbb{R})$ satisfy 
	$$
	g(\lambda(\varphi)) \leq c g(\lambda(\psi)) \; \text{on}\; (\varphi<\psi),
	$$
	for some constant $c\in [0,1)$. Then $\varphi\geq \psi$. 
\end{prop}

\begin{proof}
	Since $X$ is compact and $\varphi-\psi$ is continuous, the minimum is attained at some $x_0\in X$. By the maximum principle, we have $\lambda(\varphi) \geq \lambda(\psi)$ at $x_0$. The ellipticity condition on $g$ thus yields $g(\lambda(\varphi))\geq g(\lambda(\psi))$.  It thus follows  from the assumption that $x_0\notin (\varphi-\psi<0)$, hence $\varphi\geq \psi$. 
\end{proof}

\end{document}